\documentclass[12pt]{amsart}
\usepackage{graphicx, amssymb}
\usepackage{amsfonts}
\usepackage{amsmath}
\usepackage{hyperref}
\usepackage{geometry, amsopn}
\usepackage{comment, xcolor}

\DeclareMathOperator{\FIN}{FIN}
\DeclareMathOperator{\INF}{INF}
\DeclareMathOperator{\COF}{COF}
\DeclareMathOperator{\REC}{REC}
\DeclareMathOperator{\Sym}{Sym}
\DeclareMathOperator{\sgr}{sgr}

\newtheorem{theorem}{Theorem}

\newtheorem{cor}[theorem]{Corollary}
\newtheorem{definition}[theorem]{Definition}

\newtheorem{lemma}[theorem]{Lemma}

\newcommand{\ints}{\mathbb{Z}}
\newcommand{\nats}{\mathbb{N}}
\newcommand{\bra}{\langle}
\newcommand{\ket}{\rangle}
\renewcommand{\arraystretch}{1.2}

\numberwithin{equation}{section}

\begin{document}
\title{Detecting properties from descriptions of groups}
\author{Iva Bilanovic} 	
\author{Jennifer Chubb}
\author{Sam Roven} 	
\address{Department of Mathematics, George Washington University, Washington D.C.}
\email{ivabilanovic@gwu.edu}
\address{Department of Mathematics, University of San Francisco, San
Francisco, CA}
\email{jcchubb@usfca.edu}
\address{Department of Mathematics, University of Washington, Seattle, WA}
\email{roven@uw.edu}


\begin{abstract}

We consider whether given a simple, finite description of a group in the form of an algorithm, it is possible to algorithmically determine if the corresponding group has some specified property or not. When there is such an algorithm, we say the property  is \textit{recursively recognizable} within some class of descriptions.  When there is not, we ask how difficult it is to detect the property in an algorithmic sense.

We consider descriptions of two sorts: first, recursive presentations in terms of generators and relators, and second, algorithms for computing the group operation.  For both classes of descriptions, we show that a large class of natural algebraic properties, \emph{Markov properties}, are not recursively recognizable, indeed they are $\Pi^0_2$-hard to detect in recursively presented groups and $\Pi^0_1$-hard to detect in  computable groups.  These theorems suffice to give a sharp complexity measure for the detection problem of a number of typical group properties, for example, being abelian, torsion-free, orderable.  Some properties, like being cyclic, nilpotent, or solvable, are much harder to detect, and we give sharp characterizations of the corresponding detection problems from a number of them.  

We give special attention to orderability properties, as this was a main motivation at the beginning of this project.

\end{abstract}

%\textup{Mathematics Subject Classification 2010: 03D40, 20F10, 06F15}

\maketitle

%\section*{To do list}

%Find a place for Ash's theorem.

%\begin{enumerate}

%\item Things to add for computable groups:

%\begin{enumerate}

%\item Soluble

%\item Conjugacy problem:  Sort out miller's proof for that one case (where $S_1$ is computable) and see if $m$-degree of the conjugacy problem can be $S_2$.

%\end{enumerate}

%\end{enumerate}

%\newpage

\section{Introduction}

The complexity of the word, conjugacy, and isomorphism problems of finitely presented groups have long been of interest in combinatorial group theory and algebra in general (\cite{Dehn1911, Dehn1912}).  Questions of whether and how the presentation of a group in terms of generators and relators can shed any light on the existence of algorithms that uniformly solve these problem, or that can determine whether or not the group has some other property of interest,  have been studied, though primarily for finite presentations of groups.  Here, we generalize some of these results to \emph{recursive presentations} of groups.

Additionally, in computable structure theory, we can consider the detection of a property from another type of basic description of a group, its atomic diagram (i.e., its multiplication table).  A group is \textit{computable} when it is computable as a set and the group operation is computable, that is, finitely describable in the form of an algorithm.

From both vantage points, we are asking whether there an algorithm that may be applied uniformly \emph{for all descriptions} in some class of descriptions of groups, that will answer the question, \begin{quote}
Does the group with description $D$ have property $P$ or not?
\end{quote}
\noindent When there is such an effective procedure, we say the property is \emph{recursively recognizable} within that class of descriptions.  %When it is possible to effectively determine whether a given computable group has the property, we say that the collection of groups with that property has a \textit{computable index set}. 

In the 1950's, Adian and Rabin showed that an entire class of properties, \emph{Markov properties} (see Definition \ref{markov_def} below), are not recursively recognizable in the class of finitely presented groups (\cite{Adian1957a, Adian1957b, Rabin1958}).  Among many others, these properties include having a decidable word problem, being nilpotent, abelian, simple, torsion-free, and free. (See \cite{Miller1992} for an excellent survey.)  

In the language of the Kleene-Mostowski arithmetical hierarchy of sets, Adian and Rabin's proofs established the $\Sigma^0_1$-hardness (see Definition \ref{complete} below) of detection of Markov properties in the class of finitely presented groups.  Some of these properties are immediately seen to be $\Sigma^0_1$-complete via their $\Sigma^0_1$ characterizing formulas. For example, a finite presentation of a group yields the trivial group if and only if \textit{there exists a finite sequence of Tietze transformations} that transform the given presentation into $\bra x~|~x\ket$, which is a $\Sigma^0_1$ characterization of triviality for finitely presented groups (see, for example, \cite{MagnusCGT}). 

In the 1960's, Boone and Rogers considered questions posed by Whitehead and Church in the late 1950's \cite{BooneRogers1966}.  Whitehead asked if the collection of finite presentations of groups having decidable word problem, though it is not a recursive set, could be recursively enumerated.  Church asked about the possible existence of a universal \textit{partial} algorithm capable of solving the word problem for all finitely presented groups that have decidable word problem.  Boone and Rogers answered both questions by establishing the precise complexity of the question ``Given a finite presentation of a group, does the corresponding group have a decidable word problem?''  They showed that this question is $\Sigma^0_3$-complete in the arithmetical hierarchy, and negative answers to the questions of Whitehead and Church follow.

The precise complexity of identifying other properties from group presentations have since been studied.  For example, in the 1990's, Lempp showed that detecting torsion-freeness (a Markov property) is $\Pi^0_2$-complete in the class of finitely presented groups \cite{Lempp1997}.

In computable structure theory, recursive recognizability of a property amounts to the index set of groups that exhibit the property being a computable set relative to the set of indicies of all computable groups.  An \textit{index set} is the set of indicies (i.e., G\"{o}del codes) of computable structures of some sort.  In \cite{CHKLMMQSW}, the authors characterize the complexity of detecting rank-$k$ free groups and show that it is $d-\Sigma^0_2$ complete in the class of computable free groups. They also show that determining whether an arbitrary computable group is free is $\Pi^0_4$-complete. 

In \cite{CHKM2006}, in a similar vein, the authors determine the complexity of the set of all indices for computable isomorphic copies of a given structure (finite structures, vector spaces, Archimedean real closed ordered fields, and certain $p$-groups).  Their techniques utilize structure-characterizing Scott formulas in infinitary logic.

We were originally motivated by the question of identifying orderability properties of groups from simple (i.e., finite) descriptions, like an algorithm for an atomic diagram or for enumerating a presentation. While we do discuss our results in that context in Sections \ref{po_sec} and \ref{bo_sec}, we begin with some more general results, their immediate corollaries, and completeness results for other natural abstract properties of groups.

\section{Definitions}

A group $G$ is said to be \textit{finitely presented} if it is described by finitely many generators and relators, $\langle x_0, x_1, \ldots, x_n ~|~ R_0, R_1, \ldots, R_k\rangle$.  A group is \textit{recursively presented} (we will write \textit{r.p.}) when it is described by a computable set of generators and there is an algorithm for enumerating the (possibly infinite number of) relators (see \cite{LyndonSchupp}).  It is not hard to show that if the set of relators is recursively enumerable, then it is possible to obtain (in a uniform way) a recursive presentation on the same set of generators.

\medskip 

For purposes of characterizing complexity, we will use the following framework (note that this definition is equivalent to that given in \cite{CHKLMMQSW}).

\begin{definition}\label{complete}Let $\Gamma$ be a complexity class in the arithmetical hierarchy and $A$ an index set for some collection of recursive presentations of groups (or atomic diagrams of computable groups). We say \emph{detecting property $P$ is $\Gamma$-complete in $A$} if the following hold. 

\begin{enumerate}

\item There is a $\Gamma$ formula $\phi(e)$ so that $B= A\cap \{e\in \nats~|~ \phi(e)\}$ is exactly the set of indices of recursive presentations of groups (or atomic diagrams of computable groups) which exhibit property $P$.

\item For any $\Gamma$ set $S$, there is a computable function $f:\nats\to A$ so that $e\in S$ if and only if $e\in B$. 

\end{enumerate}

Whenever \emph{(1)} holds, we say detecting \emph{$P$ is $\Gamma$ in $A$}, and when \emph{(2)} holds, we say detecting \emph{$P$ is $\Gamma$-hard in $A$}.\end{definition}  

\noindent For example, let $A$ be the set of all indices of computable groups, and $P$ the property ``abelian''.  This property is described by the $\Pi^0_1$ formula $\forall x,y ~ xy=yx$, so the corresponding detection problem is $\Pi^0_1$ in the class of computable groups (later we will see that it is $\Pi^0_1$-complete).

\medskip

\medskip

%Of course, a recursive presentation or an algorithm $\varphi_e$ both suffice to completely describe a group, and it is perhaps worth noting that every computable group has a recursive presentation.  Recursively presented groups need not be finitely presentable, but the algorithm that computes the presentation is finite, so such groups may be thought of as finitely describable. 

\begin{definition}\label{markov_def}

An property $P$ of groups is \emph{Markov for a class} $\mathcal C$ \emph{of groups} if there is a group $G_+\in\mathcal C$ which exhibits the property, and there is a group, $G_- \in \mathcal C$,  so that for any group $H$, if there is a injective homomorphism from $G_-$ into $H$, $H$ fails to have property $P$.

\end{definition}

Rabin's Theorem (1.1 in \cite{Rabin1958}) asserts that Markov properties of finitely presented groups are not recursively recognizable in that class. It follows from the works of Collins and Lockhart(\cite{Collins1970}, \cite{Lockhart1981}) that the same holds true in the class of computable groups.  A look at the respective proofs reveals that Rabin showed $\Sigma^0_1$-hardness of Markov properties for finitely presented groups, and Lockhart showed these properties are $\Pi^0_1$-hard to detect in computable groups. 

The table below summarizes most of the results in this article.

\begin{center}

\renewcommand{\arraystretch}{1.4}

\begin{tabular}{|l|c|c|}

\hline
\textbf{Property} & \textbf{Class of r.p. groups} & \textbf{Class of computable groups} \\
\hline
Markov property & $\Pi^0_2$-hard & $\Pi^0_1$-hard (given an infinite $G_+$) \\
\hline 
Abelian & $\Pi^0_2$-complete & $\Pi^0_1$-complete \\
\hline
torsion-free & $\Pi^0_2$-complete & $\Pi^0_1$-complete \\
\hline
trivial & $\Pi^0_2$-complete & n/a \\
\hline
divisible & $\Pi^0_2$-complete & $\Pi^0_2$-complete \\
\hline
torsion & $\Pi^0_2$-complete & $\Pi^0_2$-complete \\
\hline
totally left-orderable & $\Pi^0_2$-complete & $\Pi^0_1$-complete \\
\hline
totally bi-orderable & $\Pi^0_2$-complete & $\Pi^0_1$-complete \\
\hline
finite & $\Sigma^0_3$-complete & n/a \\
\hline
decidable word problem & $\Sigma^0_3$-complete & n/a \\
\hline
cyclic & $\Sigma^0_3$-complete & in $\Sigma^0_3$ \\
\hline
nilpotent & $\Sigma^0_3$-complete & $\Sigma^0_2$-complete \\
\hline
solvable & $\Sigma^0_3$-complete & $\Sigma^0_2$-complete \\
\hline
finitely presentable & $\Sigma^0_3$-complete & in $\Sigma^0_4$ \\
\hline
\end{tabular}
\end{center}

\medskip

We use standard computability-theoretic notation throughout (as in \cite{Soare}), denoting the $e$th partially computable function on the natural numbers in some fixed, acceptable enumeration Turing machines by $\varphi_e$, and its domain by $W_e$.  $W_{e,s}$ is the $s$th finite approximation of $W_e$, and we assume throughout that the cardinality of $W_{e,s+1}-W_{e,s}$ is at most one.  A group is computable if its atomic diagram is computed by some $\varphi_e$.

\section{Decision problems in recursively presented groups}

\subsection{A general theorem}

 We begin by considering detection of Markov properties in the class of recursively presented groups, which contains both finitely presented groups and all computable groups.

In what follows, we will conflate the presentation of a group with the group itself on occasion.

\begin{theorem}\label{markov}
Let $P$ be a Markov property for r.p.~groups. Detection of $P$ is $\Pi^0_2$-hard in the class of r.p.~groups.
\end{theorem}

\begin{proof}

We reduce the set of indices of the infinite c.e.~sets, $\INF=\{e~|~|W_e|=\omega\}$ to the detection of $P$.

Let $$G_+ = \langle x_0, x_1, \ldots ~|~ R_0, R_1, \ldots\rangle=\langle \mathbf{x}~|~\mathbf{R(x)}\rangle$$ and $$G_- = \langle y_0, y_1, \ldots ~|~ S_0, S_1, \ldots \rangle = \langle \mathbf{y}~|~\mathbf{S(y)}\rangle$$ witness that $P$ is Markov for r.p.~groups as in Definition \ref{markov_def} above.  For each $e\in \nats$, we give a recursive presentation of a group $G_e$ so that $G_e$ has property $P$ if and only of $e\in \INF$.  

We will need to include multiple copies of the presentation of $G_-$ on different sets of generators, and so will write  $$G_-(\mathbf{y_i}) = \langle y_{i,0}, y_{i,1}, \ldots ~|~ S_0, S_1, \ldots \rangle = \langle \mathbf{y_i}~|~\mathbf{S(y_i)}\rangle,$$ that we may specify distinct generating sets.  Let $A\ast B$ denote the free product of groups $A$ and $B$.

~

\noindent \emph{Construction.}

\noindent\emph{Stage 0.}  Initialize by setting $$G_{e,0} = G_+\ast G_-(\mathbf{y_0}) \ast G_-(\mathbf{y_1}) \ast \cdots  = \langle \mathbf{x}, \mathbf{y_0}, \mathbf{y_1}, \ldots ~|~ \mathbf{R(x)}, \mathbf{S(y_0)}, \mathbf{S(y_1)}, \ldots \rangle,$$ and $n=0$.

\medskip

\noindent\emph{Stage $s+1$.}  The stage begins with the presentation $G_{e,s} = G_{e,0}$ if $n=0$ or, if $n>0$, the presentation $$G_{e,s} = \langle \mathbf{x}, \mathbf{y_0}, \mathbf{y_{1}}, \ldots ~|~ \mathbf{y_0}, \ldots, \mathbf{y_{n-1}}, \mathbf{R(x)}, \mathbf{S(y_0)}, \mathbf{S(y_1)}, \ldots\rangle.$$

If $W_{e,s+1}-W_{e,s}$ is empty, set $G_{e,s+1}=G_{e,s}$.  Otherwise, set $$G_{e,s+1} = \langle \mathbf{x}, \mathbf{y_0}, \mathbf{y_{1}}, \ldots~|~ \mathbf{y_0}, \ldots, \mathbf{y_{n}}, \mathbf{R(x)}, \mathbf{S(y_0)}, \mathbf{S(y_1)}, \ldots \rangle,$$ and increment $n$.

Let $G_e$ be the limit $\lim_s G_{e,s}$.

\noindent \emph{End of construction.}

\medskip

It is easy to see that $G_e$ is a r.p.~group.  If $W_e$ is of finite cardinality $n$, then $G_e$ is the group with presentation $$G_{e} = \langle \mathbf{x}, \mathbf{y_0}, \mathbf{y_{1}}, \ldots ~|~ \mathbf{y_0}, \ldots, \mathbf{y_{n-1}}, \mathbf{R(x)}, \mathbf{S(y_0)}, \mathbf{S(y_1)}, \ldots \rangle,$$ which is isomorphic to $G_+\ast G_-\ast G_- \ast \cdots $, so contains $G_-$ as a subgroup and therefore does not have property $P$.  

If $W_e$ is infinite, the presentation that results from the construction is $$G_{e} = \langle \mathbf{x}, \mathbf{y_0}, \mathbf{y_1}, \ldots ~|~ \mathbf{y_0}, \mathbf{y_1}, \ldots, \mathbf{R(x)}, \mathbf{S(y_0)}, \mathbf{S(y_1)}, \ldots \rangle,$$ which is isomorphic to $G_+$ and so does have property $P$.

\end{proof}

It follows that detection of any Markov property that can be characterized by a finitary $\Pi^0_2$ formula, or by a computable infintary $\Pi_2$ formula, is a $\Pi^0_2$-complete decision problem (see  \cite{AshKnight}, especially Theorem 7.5, for more on computable infintary formulas). Some examples of such properties are included in the corollary below.

It should be noted that in a r.p.~group, equality is not generally a computable predicate, but it is $\Sigma^0_1$ (and inequality is consequently $\Pi^0_1$).  Throughout, we will write $=_G$ to denote equality in the group $G$, $F_G$ for the free group on its generators, and $1_G$ for its identity. 

It is easy to see that words in $F_G$ which evaluate to the identity in group $G$ can be algorithmically enumerated.  So, when we write ``$w=_G v$'' for some $w,v\in F_G$, we mean that ``$\exists s\in ~ \nats~ wv^{-1} \in 1_{G,s}$'', where $1_{G,s}$ is the $s$th finite approximation of the set of words in $F_G$ which evaluate to the identity in $G$.

\begin{cor}\label{cor_rpm}
Detection problems for the following properties are $\Pi^0_2$-complete in the class of recursively presented groups.

\begin{itemize}

\item Being abelian.
\item Being torsion-free.
\item Being trivial.
\item Being divisible.
\item Being a torsion group (in which all elements have finite order).
\item Being totally left- or bi-orderable (see Sections \ref{po_sec} and \ref{bo_sec} below).

\end{itemize}

\end{cor}

\begin{proof}

Abelian groups are precisely those $G$ which satisfy the formula $$\forall w \in F_G~ \forall v \in F_G~ wv =_G vw,$$ which is $\Pi^0_2$ as $=_G$ is a $\Sigma^0_1$ predicate. Set $G_+ = \langle x ~|~ \rangle$ and $G_- = \langle x, y ~|~ \rangle$ and apply the theorem. 

\medskip

Torsion-free groups are characterized by the formula $$\forall w\in F_G~ \forall n \in\nats~  (w=_G 1_G\vee w^n \neq_G 1_G),$$ which is $\Pi^0_2$ (again due to the fact that equality is enumerable).  Set $G_+ = \langle x~|~\rangle$ and $G_- = \langle y ~|~ y^2 \rangle$ and apply the theorem.

\medskip
For triviality, the characterizing formula is $\forall w\in F_G ~ w=_G 1_G$, which is $\Pi^0_2$.  Let $G_+ = \langle	 x~|~ x \rangle$ and $G_- = \langle y ~|~ \rangle$ and apply the theorem.

\medskip

Group $G$ is divisible if and only if $$\forall w \in F_G~ \forall n \in \nats^{>0}~ \exists  v \in F_G~ ( w =_G1_G \vee v^n =_G w),$$ which is again $\Pi^0_2$. Set $G_+ = \langle x_1, x_2, \ldots ~|~ x^p_1 = 1, x^p_2 = x_1, x^p_3 = x_2, \ldots \rangle$, the Pr{\"u}fer group, $\mathbb{Z}(p^\infty)$, for some prime $p$, and $G_ - = \langle x ~|~ \rangle$ and apply the theorem. 

\medskip

Torsion groups are characterized by the formula $$\forall w\in F_G~ \exists n\in\nats~ (w^n=_G1_G),$$ which is $\Pi^0_2$.  Set $G_+=\langle x ~|~ x^2\rangle$ and $G_- = \langle y~|~\rangle$ and apply the theorem.

\medskip

The characterization of orderability appears in Section \ref{po_sec} below.

\end{proof}

\subsection{Completeness at higher levels of the arithmetical hierarchy}

\begin{theorem}

Detecting finiteness is $\Sigma^0_3$-complete in r.p.~groups.

\end{theorem}

\begin{proof}

A $\Sigma^0_3$ formula characterizing finiteness is $$\exists n\in\nats ~\exists \{w_1, \ldots, w_n\}\in (F_G)^n~ \forall v\in F_G~ v\in_G \{w_1, \ldots, w_n\},$$ where ``$\in_G$'' abbreviates the $\Sigma^0_1$ formula saying that $v$ is equal (in $G$) to one of the $w_i$'s.    

To show completeness, we reduce $\COF=\{ e\in \nats ~|~ |\overline{W_e}|<\omega\}$ to the detection problem.  For all $e\in\nats$, set
$$G_e = \langle x_0, x_1, \ldots ~|~ [x_i, x_j], x_i^2, \textrm{ for all } i,j\in \nats,\textrm{ and } x_k \textrm{ for all }k\in W_e\rangle,$$ where $[x,y]$ denotes the commutator of group elements $x$ and $y$.  Now, if $W_e$ is cofinite with $|\overline{W_e}| = n$, we have $G_e\cong \ints_2^n$.  If it is not, $G_e\cong \ints_2^\omega$.

\end{proof}

\begin{theorem}

Detecting a group with a decidable word problem is $\Sigma^0_3$-complete in r.p.~groups

\end{theorem}

\begin{proof}
Let $G=\langle x_1, x_2, \ldots ~|~ R_1, R_2, \ldots\rangle$ be a r.p.~group.

The property \emph{``has a decidable word problem''} is characterized by the $\Sigma^0_3$ formula $$\exists e \in \nats~ \forall w \in F_G~ \exists s \in \nats~ (\varphi_{e,s}(w) \downarrow \wedge~ (\varphi_{e,s}(w) = 1 \leftrightarrow w =_G 1_G)).$$   Accounting for enumerability of equality and rewriting in prenex normal form yields an equivalent formula more easily seen to be $\Sigma^0_3$,

$$\exists e\in \nats ~ \forall w \in F_G ~ \forall t_2 \in \nats ~\exists t_1\in\nats ~\exists s\in \nats $$ $$\left[(\varphi_{e,s}(w)\downarrow) \wedge ~ \left(\varphi_{e,s}(w)\neq 1 \vee  w\in 1_{G,t_1}\right) \wedge \left(\varphi_{e,s}(w)= 1 \vee  w\not\in 1_{G,t_2}\right) \right].$$

For completeness, consider $G_e = \langle a, b, c, d ~|~ a^n b a^n =_G c^n d c^n, n \in W_e \rangle$. The group $G_e$ has a decidable word problem if and only if $e$ is in the $\Sigma^0_3$-complete set $\REC = \{ e \in \nats ~|~ W_e$ is recursive$\}$.

%The conjugacy problem can be characterized as a first order formula most naturally as follows:
%$$\exists e\in \nats ~ \forall u,v \in F ~ \exists s\in \nats ~ \left[\varphi_{e,s}(u,v)\downarrow \wedge ~ \left(\varphi_{e,s}(u,v)=1 \leftrightarrow \exists w\in F_G ~ wuw^{-1}=_G v\right) \right].$$ Accounting for enumerability of equality and rewriting in prenex normal form yields an equivalent formula more easily seen to be $\Sigma^0_3$ (we will write ``$w=_{G,s} v$'' for ``$wv^{-1}\in 1_{G,s}$''):
%$$\exists e\in \nats ~ \forall u,v \in F_G ~ \forall t_2 \in \nats ~\forall w_2 \in F ~\exists t_1\in\nats ~\exists w_1 \in F ~\exists s\in \nats $$ $$\left[(\varphi_{e,s}(u,v)\downarrow) \wedge ~ \left(\varphi_{e,s}(u,v)=1 \to  w_1uw_1^{-1}=_{G,t_1}v\right) \wedge \left(\varphi_{e,s}(u,v)\neq 1 \to  w_2uw_2^{-1}\neq_{G,t_2}v\right) \right].$$  

%Since the conjugacy problem trivially computes the word problem, we have completeness.

\end{proof}

\begin{theorem} 
Detecting a cyclic group in the r.p.~groups is $\Sigma^0_3$-complete. 
\end{theorem}

\begin{proof}

The property of being a cyclic group is characterized by the $\Sigma^0_3$  formula $$\exists w \in F_G~ \forall v \in F_G~ \exists n \in \nats^{>0}~ (v =_G1_G \vee w^n =_G v).$$ For completeness we reduce $\COF$ as follows, making use of the fact that the product of cyclic groups of the form $\ints_n$ and $\ints_m$ is cyclic if and only if $n$ and $m$ are relatively prime.  Let $p_n$ be the $n$th prime number.  For each $e\in\nats$, we enumerate a presentation of $G_e$ as the limit of groups $G_{e,s}$.

\medskip

\noindent\emph{Construction.}

\noindent \textit{Stage 0.}  Initialize $G_{e,0} = \left\langle x_0, x_1, \ldots ~|~ x_0^{p_0}, x_1^{p_1}, \ldots, \textrm{ and } \forall i,j \in \nats ~[x_i, x_j]\right\rangle.$
 
\smallskip 
 
\noindent \textit{Stage $s+1$.}  If $W_{e,s+1}-W_e = \emptyset$, set $G_{e,s+1}=G_{e,s}$.  Otherwise, if $n\in W_{e,s+1}-W_e$, add $x_n$ to the relators of the presentation of $G_{e,s}$ to obtain the presentation for $G_{e,s+1}$.

\smallskip

Let $G_e = \lim_{s} G_{e,s}$.

\noindent\emph{End of construction.}

\medskip

The construction gives a recursive presentation for a group which is a finite direct sum of cyclic groups of relatively prime orders if the set $W_e$ is cofinite, and is an infinite direct sum of such groups otherwise.

\end{proof}

\begin{theorem}\label{rp_nil}
Detecting a nilpotent group is  $\Sigma^0_3$-complete in r.p.~groups (even in the class of r.p.~residually nilpotent groups).  
\end{theorem}

\begin{proof}

Recall that a group $G$ is nilpotent if it has a central series of finite length.  We consider the lower central series of $G$, that is $$G=G_0 \geq G_1 \geq \ldots \geq G_n = \{1_G\},$$ where for each $i\leq n$, $G_{i+1} = [G_i,G]$.  The finiteness of this series can be expressed as a $\Sigma^0_3$ formula as  $$\exists n\in \nats ~ \forall \vec{g} \in G^n ~ [[...[g_0,g_1],g_2]...],g_n]=_G1_G,$$ where $[x,y]$ denotes the commutator $x^{-1}y^{-1}xy$.

For completeness we build a presentation for a group $G_e$ in stages so that $W_e$ is cofinite if and only if $G_e$ is nilpotent. 

Let $\{H_n\}_{n>0}$ be a sequence of uniformly r.p.~nilpotent groups with strictly increasing nilpotency class,  $H_n$ is has nilpotency class $n$.  For example, we can set $H_n = \ints_{p_{n}} \wr \ints_{p_{n}}$, where $\wr$ denotes the wreath product,\footnote{For groups $G$ and $H$ and the left group action $\rho$ of $H$ on itself, the \emph{regular wreath product} of $G$ by $H$ is the semidirect product $G^H \rtimes H$ where $G^H$ is the direct sum of $|H|$-many copies of $G$. The \emph{regular wreath product} is denoted $G \wr H$. } and $p_{n}$ is the $n$th prime
number.  It is well known that for any prime $p$ the regular wreath product $\mathbb{Z}_p \wr \mathbb{Z}_p$ is isomorphic to the Sylow $p$-subgroup of the symmetric group $\Sym(p^2)$ and has nilpotency class $p$. Moreover, as these groups are finite, they have finite presentations.

For each $n$, we take a finite presentation for $H_n=\langle a_{n,1}, \ldots, a_{n,k_n} ~|~ R_{n,1}, \ldots, R_{n,j_n}\rangle$.

~

\noindent\emph{Construction}.

\noindent\emph{Stage 0.}  Begin with $G_{e,0}$ as the direct sum, $\bigoplus_{n\in\omega} H_n$ given by presentation $$\langle a_{m,k}\textrm{ for } m\in \nats, k\leq k_m ~|~ R_{m,j}\textrm{ for } m\in \nats,  j\leq j_m,\textrm{ and } [a_{n,k}, a_{m,j}] \textrm{ for } n\neq m\rangle .$$ 

\smallskip

\noindent \emph{Stage $s+1$.}  When $n \in W_{e,s+1} - W_{e,s}$, enumerate the generators of $H_n$ into the relators of $G_{e,s}$ to obtain $G_{e,s+1}$. 

If $W_{e,s+1} - W_{e,s}=\emptyset$, no action is required.

Let $G_e=\lim_{s} G_{e,s}$.

\noindent\emph{End of construction.}

\medskip

This construction yields a recursive presentation of group $G_e$, which is the direct sum of finitely many nilpotent groups, and thus itself nilpotent, provided that $e \in \COF$. If $e\not\in \COF$, then $G_e$ is residually nilpotent but not nilpotent as it contains subgroups of arbitrarily large nilpotency class.  

\end{proof}

\begin{cor}Determining whether a r.p.~group is finitely presentable is $\Sigma^0_3$-complete. 

\end{cor}

\begin{proof}

$\Sigma^0_3$-hardness follows immediately from the proof of Theorem \ref{rp_nil}, since $G_e$ is finitely presentable if and only if $e\in \COF$. Being finitely presentable is characterized (informally) by the statement below.  Let $G=\langle x_1, x_2, \ldots ~|~ R_1, R_2, \ldots \rangle$.  We write $\overline{g}(\overline{x})$ for a finite sequence of words in the generators and their inverses, $\{x_i^{\pm 1}\}_{i\in\nats}$, and $w(\overline{g})$ for a word on the elements of $\overline{g}$ and their inverses, and $\overline{w}(\overline{g})$ for a sequence of such words. 

\noindent Now, group $G$ is finitely presentable if and only if the following $\Sigma^0_3$ formula holds. $$\begin{array}{c}(\exists \overline{g}(\overline{x})\in F_G^{<\omega})(\exists \overline{w}(\overline{g})\in F_{\overline{g}})(\forall h\in F_G)(\forall u,v\in F_G)(\forall s,t\in \nats) (\exists s', t' \in \nats)\\ (h=_G v) \wedge (u\in 1_{G,s} \to u\in 1_{\overline{g},s'})\wedge ( v\in 1_{\overline{g},t} \to v\in 1_{G,t'}).\end{array}$$  The theorem follows.

\end{proof}

%*****************************************

%One such sequence $\{H_n\}_{n \in N}$ of r.p. groups of ever increasing nilpotent classes are the Sylow $p_n$-subgroups of the symmetric groups $S_{p_n^2}$ for primes $p_n$. The standard notation \textcolor{red}{('Notation' is not the word I want. What is?)} for these groups is that of the \textit{regular wreath product}.

\medskip

%For the construction above, begin with $G_{e,0} = H_0 \cong \mathbb{Z}_2 \wr \mathbb{Z}_2$. Whenever $W_{e,s+1} - W_{e,s} \neq \emptyset$ let $G_{e,s+1} = G_{e,s} \times H_n$ where $p_n$ is the least prime such that $H_n$ has not appeared in the construction. Otherwise $G_{e,s+1} = G_{e,s}$. 

%$G_e$ is nilpotent if and only if $W_e$ is cofinite. 

%********************************************

\begin{theorem}

Detecting a solvable group is $\Sigma^0_3$-complete in the class of r.p.~groups, and even in the class of residually solvable r.p.~groups.

\end{theorem}

\begin{proof}

Recall that group $G$ is solvable if its derived series is finite.  That is, there is an $n\in \nats$ so that $$G=G_0\geq G_1 \geq \cdots \geq G_n=\{1_G\},$$ where for each $i< n$, $G_{i+1} = [ G_i,G_i]$.  In the language of first-order logic, we can write  $\exists n\in \nats ~ \forall \vec{g}\in G^{2^n} $, 
$$ \left[\cdots[[g_1, g_2],[g_3,g_4]],\cdots],[\cdots , [[g_{2^n-3},g_{2^n-2}],[g_{2^n-1},g_{2^n}]\right]\cdots] =_G 1_G,$$ i.e., every $n$-deep commutator of the correct form evaluates to the identity in $G$.  This is a $\Sigma^0_3$ formula (the matrix as shown is $\Sigma^0_1$), so it remains to show completeness.

As in the construction in the proof of Theorem \ref{rp_nil}, we will enumerate a presentation of a group $G_e$ which, when $e\in \COF$, is isomorphic to a direct sum of finitely many solvable groups, so is itself solvable.  When $e\not\in \COF$, $G_e$ will contain subgroups of arbitrarily large solvability class.

For each $n\in \nats$, let $H_n$ be the \textit{free solvable group} of rank 2 and class $n$.  That is, $H_n$ will be the quotient of the free group $F_2$ by its $n$th derived subgroup.  $H_n$ is computable uniformly in $n$ (\cite{MRUV}), so has a recursive presentation.  

As in the previous construction, we begin the construction with a presentation of the direct sum $G_{e,0}=\bigoplus_{n\in\omega} H_n$. Whenever $n\in W_{e,s+1} - W_{e,s}$, we enumerate the generators of $H_n$ into the relators of our presentation of $G_e$.  
 
If $e\in \COF$, $G_e$ has solvability class $\max(\overline{W_e})$, and otherwise is residually solvable but not solvable.

\end{proof}

\section{Decision problems in the class of computable groups.}

\subsection{General theorem and immediate consequences}

That decision problems for Markov properties in the class of computable groups are $\Pi^0_1$-hard follows from work of Lockhart and Collins (\cite{Lockhart1981, Collins1970}).  Here, we give a new proof which is entirely constructive. 

\begin{theorem}\label{markov_comp} 
Let $P$ be a Markov property for computable groups and let $G_+$ and $G_-$ witness that $P$ is Markov as in definition \ref{markov_def}. Then detection of $P$ is $\Pi^0_1$-hard in the class of computable groups.
\end{theorem}

\begin{proof} Note that we may as well assume $G_-$ is infinite since $G_-$ is a subgroup of the direct sum of itself with the (computable) additive group of integers, $G_-\times \ints$.  This product is necessarily computable and fails to have property $P$ by definition.

The strategy here is to use the computable atomic diagrams of $G_+$ and $G_-$ to build uniformly in $e$ a computable atomic diagram of a group $G_e$ so that $G_e\cong G_+$ if $\varphi_e(e)\uparrow$, and $G_e \cong G_+ \times G_-$ if $\varphi_e(e)$ eventually halts.  If we can manage this, we will have $$e\in \overline{K} \iff G_e\models P.$$  Since $\overline{K}$ is a $\Pi^0_1$-complete set, it follows that detecting $P$ is $\Pi^0_1$-hard.

Let $G_+ = \{g_0=1_+, g_1, \ldots\}$ and $G_- = \{h_0=1_-, h_1, \ldots\}$ be enumerations of the groups witnessing that $P$ is Markov without repetitions.  The universe of $G_e$ will be $\nats$, and we give a coding map, $\langle \cdot \rangle$, and enumerate the atomic diagram in stages below.

\medskip

\emph{Construction.}

\emph{Stage 0.}  Let $\langle (g_0, h_0)\rangle = 0$, and add $(0,0,0)$ to the atomic diagram indicating that $(g_0, h_0) \ast (g_0, h_0) = (g_0, h_0)$ in the group we are constructing.

\smallskip

\emph{Stage $s+1$.}  We begin this stage with a coding map having an initial segment of the natural numbers as its range, and a finite set of triples.  There are three cases.

\begin{enumerate}

\item $\varphi_{e,s+1}(e)\uparrow$.  Let $i$ be the least index of an element of $G_+$ for which $(g_i, h_0)$ has not yet been assigned a code, and assign it the least available code.  Next, let $j$ be the least index of an element of $G_+$ for which there exists a $k\leq j$ such that there is no tuple yet in the diagram for  $G_e$ indicating the product $(g_j,h_0) \ast (g_k, h_0)$.  For each such $k\leq j$, assign both $(g_jg_k, h_0)$ and $(g_kg_j,h_0)$ codes (if necessary), and add the corresponding triples to the diagram.  For example, if $g_j g_k = g_n$ in $G_+$, and $(g_n, h_0)$ does not already have a code, we assign one to it, say $m$, and add the triple $(\langle (g_j, h_0)\rangle, \langle (g_k, h_0)\rangle, m)$ to the diagram.

\item $\varphi_{e,s+1}(e)\downarrow$ but $\varphi_{e,s}(e)\uparrow$. This is the exact stage where $e$ enters the halting set.  After we have executed this stage once, all subsequent stages will be instances of case (3).  

So far, we have a partial diagram of a copy of $G_+\times \{1_- = h_0\}$.  We now begin to build out $G_-$ in the second coordinate.  Assign fresh natural number codes systematically to $\langle (g_i, h_1)\rangle$ for all $i>0$ for which $(g_i, h_0)$ has been assigned a code, and add tuples to the diagram accordingly.  (So, for example, if $17=\langle (g_5, h_0) \rangle$, $39 = \langle (g_4, h_1) \rangle$, $g_5g_4 = g_{11}$ in $G_+$, and $\langle (g_{11}, h_1) \rangle = 65$, we'd add $(17, 39, 65)$ to the diagram.)

\item $\varphi_{e,s}(e)\downarrow$. Here, $e$ entered the halting set at some previous stage.  Let $i$ and $j$ be the least indices for which $(g_i, h_0)$ and $(g_0, h_j)$ have not been assigned codes, and assign them codes.  Add all tuples of the form $$(\langle(g_u, h_v)\rangle, \langle(g_x, h_y)\rangle, \langle(g_ug_x, h_vh_y)\rangle)$$ for $u,x \leq i$ and $v,y\leq j$ (assigning codes as needed) to the diagram of $G_e$.

\end{enumerate}

\emph{End of construction.}

\medskip

It is clear from the construction that the group is computable, and that when $e\in\overline{K}$, $G_e$ is isomorphic to $G_+$, and so has property $P$.  If $e\in K$, $G_e \cong G_+ \times G_-$, and will fail to exhibit $P$.
\end{proof}

\begin{cor}

Detection of the following properties is $\Pi^0_1$-complete in the class of computable groups. 

\begin{itemize}
\item Being abelian.
\item Being torsion-free.
\item Being totally left- or bi-orderable.
\end{itemize}

\end{cor}

\begin{proof}

The characterizing formulas of these properties given in the proof of Corollary \ref{cor_rpm} become $\Pi^0_1$ since equality (and inequality) is computable. Moreover, since finite groups and free groups have computable copies and free groups are infinite, we can take computable instances of the same witnesses as before and apply Theorem \ref{markov_comp}.

\end{proof}

\subsection{Completeness at higher levels of the arithmetical hierarchy}

\begin{theorem}\label{thm_tor_cg}

Detecting torsion groups is $\Pi^0_2$-complete in the class of computable groups. 
\end{theorem}

\begin{proof}

The formula $\forall g\in G ~\exists n \in \nats^{>0} ~ (g^n = 1_G)$ is a $\Pi^0_2$ formula characterizing torsion groups. 

\medskip

To show completeness, we reduce the $\Pi^0_2$-complete set, $\INF=\{e\in \nats ~|~ |W_e|=\omega\}$, of indices of finite sets to the index set of non-torsion groups.  We construct for each $e\in\nats$, a computable abelian group $G_e$ that is non-torsion if and only if $W_e$ is finite.  At each stage $s$, we give a set $G_{s} = \{0, x_{\pm 1}, x_{\pm 2},\ldots  x_{\pm k_s}\}$ of natural numbers indexed by integers as the $s$th approximation of $G_e$.  We will index elements of the universe by integers, but in the end, the universe of the group will be the natural numbers.

The group we build will be isomorphic to a group of the form
$$\ints_{n_1}\times \ints_{n_2}\times\cdots \times \ints_{n_ k} \times \cdots $$  if $W_e$ is infinite, or of the form  $$\ints_{n_1}\times \ints_{n_2}\times \cdots \times \ints_{n_ k} \times \ints$$ if $W_e$ is finite.  The values $n_i$ will be determined by the stages that elements appear in the enumeration of $W_e$.  For example, if the $m$th element appears at stage $s$ and $(m+1)$st element appears at stage $t$, then $n_{m+1}=2^{(t-s)+1}$.

%[I think the group will be isomorphic to the image of a $\Delta^0_2$ homomorphism with domain $\ints^\omega$.] 

Each of the $x_i$'s will behave like a tuple of integers under coordinate-wise addition. We will arrange that the inverse of $x_i$ is $x_{-i}$ for each $i$.  To simplify discussion, we will denote the tuple  assigned to the natural number $x_i$ by $[x_i]$.  

So, for example, if the 17th and 18th elements added to the group are to ``behave like'' $(0,1,2)$ and $(0,-1,-2)$, we would have $[17]=[x_j] = (0,1,2)$ and $[18]=[x_{-j}] = (0,-1,-2)$ for some $j$, and observe that in our group, $17 + 18 = 0$, since we will set $[0]=(0)$.

When we speak of computing sums of tuples of different lengths, we will assume appended padding zeros at the end of the shorter tuple as necessary, i.e., $(2,3,4) + (1,3,0,5,6) = (3,6,4,5,6)$, modulo $n_i$ in the $i$th component.  The length of a tuple is the length of the sequence up to the last non-zero entry (e.g., the length of $(0,2,45,-11,0,0,\ldots)$ is 4).

At any given moment in the construction, we will have an element $x$ that has not been assigned a finite order, so has the potential to wind up being a non-torsion element in the end.  Whenever a new element enters $W_{e,s+1}$, we assign a finite order to $x$ by declaring a multiple of it to be the identity in such a way that we do not interfere with any sums previously declared.  Any time we add a new element to the group, we will assign it and its inverse names, $x_i$ and $x_{-i}$ for some $i\in\ints$.

\medskip

\noindent\emph{{Construction.}}

\noindent \emph{Stage 0.} We will use $x_0=0$ as the identity for our group, and begin the construction with $G_0 = \{x_0=0,x_1=1,x_{-1}=2\}$ where $[x_0]=[0]=(0), [x_1]=[1]=(1),$ and $[x_{-1}]=[2]=(-1)$.  In what follows, we at times conflate the natural number $x_i$ and the tuple $[x_i]$.

\smallskip

\noindent \emph{Stage $s+1$.} We begin this stage with $G_s = \{x_0, x_{\pm 1}, x_{\pm 2},\ldots  x_{\pm k_s}\}$, and each of these is mapped to some finite tuple of integers via the square bracket function.  Let $n$ be the length of the longest tuple(s) in $G_s$, and let $m$ be the largest positive value of the $n$th components of elements of $G_s$.  There are two cases:

\emph{Case 1.} $W_{e,s+1}-W_{e}=\emptyset$. In this case, we extend $G_s$ to $G_{s+1}$ by computing the coordinate-wise sums of all pairs of tuples in $G_s$, and assign fresh $x_i$'s to sums that are not already in $G_s$ as needed.  Note that the value in the $n$th component of the resulting sums will be no more than $2m$ and no less than $-2m$. 

\emph{Case 2.}  $W_{e,s+1}-W_e\neq \emptyset$.  When this is the case, we need to introduce torsion.  To do this, we extend $G_s$ to $G_{s+1}$ by adding sums of pairs of tuples in $G_s$ using modulo $4m$ addition in the $n$th component, but ``shifted'' by $2m$ from the usual notation.  So, for example, if $m$ is 4, then we shall perform additions modulo 16, but shifted by 8 (so 5+7 is -4 rather than 12) to avoid having to change the square bracket function.  In the end, we have the values in the $n$th components between $-2m$ and $2m-1$ only, and all subsequent additions in this component will be carried out modulo $4m$ in this manner.  %(Strictly speaking, it is not necessary to use the shift, as the multiplication table for the group is unaffected by it.  But the square bracket function would change.) 

In Case 2, we also add two new tuples of length $n+1$, $(0,\ldots,0,1)$ and $(0,\ldots,0,-1)$ to $G_{s+1}$, and take pair-wise sums as described in Case 1.

Let $G_e = \bigcup_s G_s$.

\noindent{\emph{End of construction.}}

\medskip

We finish the proof of the theorem with a sequence of lemmas.

~

\begin{lemma}
$G_e$ is a computable group.

\end{lemma}

\begin{proof}

It is clear that $G_e$ is a group.  To compute the sum of $x_j$ and $x_k$, execute the construction to the stage $s$ where both values have been added to $G_s$.  At stage $s+1$, their sum will be defined (and of course will not be changed later).

\end{proof}

\begin{lemma}
If $e\in \FIN$, then $G_e$ has a non-torsion element.
\end{lemma}

\begin{proof}
If $e \in \FIN$, then there is a stage $s$ so that $W_{e,s} = W_{e,s'}$  for all stages $s'\geq s$. From that stage on, only Case 1 in the construction will be executed, and the result is a group isomorphic to $$\ints_{n_1}\times \ints_{n_2}\times \cdots \times \ints_{n_ k} \times \ints$$ for some $\{n_1, \ldots, n_k\} \subset \nats$ where $k$ is the cardinality of $W_e$.
\end{proof}

\begin{lemma}

If $e\in \INF$, then $G_e$ is a torsion group.

\end{lemma}

\begin{proof}

If $e \in \INF$ then there are infinitely many stages $s$ for which $W_{e,s}\neq W_{e,s+1}$ so the construction will execute Case 2 infinitely often.

Let $x$ be a natural number that enters the group at stage $s$ and $[x] = (x_1, x_2, \ldots, x_n)$ be of length $n$ (so $x_n\neq 0$).  Note that for each $i<n$, the order of $(0,\ldots, 0, x_i, 0, \ldots, 0)$ is determined by the stages $t$ and $t'$, at which the $(i-1)$st and $i$th elements entered $W_e$, in particular, the order divides $o_i=2^{(t'-t)+1}$.  

Now let $s'>s$ be the least so that $W_{e,s}\neq W_{e,s'}$.  In this stage, Case 2 will be executed, so the element $(0, \ldots, 0, x_n)$ will have finite order that divides $2^{(s'-s'')+1}$, where $s''<s$ is the largest so that $W_{e,s''}\neq W_{e,s}$.  

At the end of stage $s'$, the element $x$ has finite order that divides $o_1o_2\cdots o_n$. 

\end{proof}

We have shown that $e \in \INF$ if and only if $G_e$ is a torsion group, and the proof is complete.

\end{proof}

\begin{theorem}
Detecting divisibility is $\Pi^0_2$ -complete in the class of computable groups.
\end{theorem}

\begin{proof}

For computable groups, the characterizing formula for divisibility is still $\Pi^0_2$: $$\forall g \in G~ \forall n \in \nats^{>0}~ \exists  h \in H~ ( g =1_G \vee h^n = g).$$ 

For completeness we reduce $\INF$, the $\Pi^0_2$-complete set of indices of infinite c.e.~sets, to the index set of divisible groups. For each $e$ we construct the atomic diagram $M_e$ of a group $G_e$ isomorphic to $(\mathbb{Q}, +)$ if $e \in \INF$ and isomorphic to some non-divisible additive subgroup of $\mathbb{Q}$ otherwise.

\medskip

The domain of the atomic diagram, $dom(M_e)$, will be $\{g_0, g_1, \ldots\}=\nats$ where each $g_i$ is the name assigned to some rational number $[g_i]$. The atomic diagram $M_e$ will be a set of triples $(g_i, g_j , g_k)$ where $g_k$ is the name assigned to the rational number $[g_i] + [g_j]$. 

\medskip

\noindent\emph{{Construction.}}

\noindent \emph{Stage 0:} Set $g_0=0$, $g_1=1$, and $g_2=-1$. Begin the construction of the atomic diagram with all triples $(g_i, g_j , g_k)$ for which no new name must be defined. That is, $M_{e,0}$ is the set of triples, $$\{(g_0, g_0, g_0), (g_0, g_1, g_1), (g_1, g_0 , g_1), (g_0, g_2, g_2), (g_2, g_0, g_2), (g_1, g_2, g_0), (g_2, g_1, g_0) \}.$$

\smallskip

\noindent \emph{Stage s+1:}  We begin this stage with $ dom(M_{e,s}) = \{ g_0, g_1, \ldots g_{n_s} \}$ and some set of triples $M_{e,s}$. First extend $M_{e,s}$ to $M_{e,s+1}$ by adding the triples $(g_i, g_j , g_k)$ for all $g_i, g_j$ already in the domain of $M_{e,s}$ assigning new names as needed. 

Next, if $W_{e,s+1} - W_{e,s} = \emptyset$, proceed to the next stage. If $W_{e,s+1} - W_{e,s} \neq \emptyset$, assign the next available name $g_n$ to the rational $\frac{1}{m}$, where $m=|W_{e,s+1}|$.

\smallskip

Set $M_e = \bigcup_s M_{e,s}$. 

\noindent\emph{{End of construction.}}

\medskip

Observe, by the construction $M_e$ and $dom(M_e)$ are computable. 

If $e \in \INF$, the resulting group is a computable copy of $(\mathbb{Q}, +)$, a divisible group. If $e \notin \INF$, no element of the group is divisible by any $n >n_e = |W_e|$  and we have a computable copy of the subgroup of $\mathbb{Q}$ generated by $\{1,\frac{1}{2}, \ldots \frac{1}{n_e}\}$.  The theorem follows. 

\end{proof}

\begin{theorem}
Detecting nilpotency is $\Sigma^0_2$ -complete in the class of computable groups.
\end{theorem}

\begin{proof}

For computable groups, the characterizing formula for nilpotency is $\Sigma_2^0$, $$\exists n\in \nats^{\geq 2} ~ \forall \vec{g} \in G^n ~ [[...[g_0,g_1],g_2]...],g_n]=_G1_G,$$ where $[x,y]$ denotes the commutator $x^{-1}y^{-1}xy$.

For completeness we reduce $\FIN$, the $\Sigma_2^0$-complete set of indices of finite c.e.~sets, to the index set of the nilpotent groups. For each $e$ we construct the atomic diagram $M_e$ of a group $G_e$ which is nilpotent if and only if $e \in \FIN$.

\medskip

Let $W(n) = \mathbb{Z}_{p_n} \wr \mathbb{Z}_{p_n}$, where $\wr$ denotes the wreath product and $p_n$ is the $n$th prime number.   We note here that $W(n)$ has nilpotency class $p_n$ and that it is finite.

Our construction will yield a computable group $G_e$ that is a direct sum of the additive group of integers and $W(n)$'s so that

$$
G_e \cong \left\{
        \begin{array}{ll}
            \mathbb{Z} \times W(1) \times \ldots \times W(n) & \quad |W_e| = n \\
             \mathbb{Z} \times W(1) \times \ldots \times W(n) \times \ldots & \quad |W_e|=\omega
        \end{array}
    \right..
$$  If $e\in \FIN$, $G_e$ will be nilpotent of class $p_n$, and residually nilpotent otherwise.

%Let $w_{ij}$ be the element $(0, 1, \ldots, 1, w_i, 1, \ldots )$ of $G_e$ where $w_i$ is some non-identity element of $W(j-1)$.  
%\medskip

The domain of $G_e$ will be $\{g_0, g_1, \ldots \}=\nats$, and we will approximate its diagram $M_e$ by finite extension. We simultaneously build the isomorphism, and denote by $[g_i]$ the tuple to which it corresponds.  For each $n\geq 1$, we write $1_n$ for the identity in $W(n)$.

\medskip

\noindent\emph{{Construction.}}

\noindent \emph{Stage 0:} Set $g_0$ as the identity of $G_e$, that is $[g_0] =(0,1_1,1_2, \ldots)$. To begin building the copy of the integers in the first component set $[g_1] =(1,1_1, 1_2, \ldots)$ and $[g_2] =(-1,1_1,1_2, \ldots)$. Begin the construction of the atomic diagram with the set of triples for which no new names must be assigned. So $M_{e,0}$ is the set of triples,  $$\{ (g_0, g_0 , g_0), (g_0, g_1, g_1), (g_1, g_0, g_1), (g_0, g_2 , g_2), (g_2, g_0, g_2), (g_1, g_2 , g_0), (g_2, g_1, g_0) \}.$$

\smallskip

\noindent \emph{Stage s+1:} We begin this stage with $dom(M_{e,s}) = \{g_0, g_1, \ldots, g_m\}$ and some set of triples $M_{e,s}$. 

First, extend $M_{e,s}$ with triples $(g_i, g_j , g_k)$ for all $g_i, g_j \in dom(M_{e,s})$, assigning new names for $g_k$ as needed.

\emph{Case 1.}  If $W_{e,s+1} - W_{e,s} = \emptyset$, proceed to the next stage.

\emph{Case 2.} If $W_{e,s+1} - W_{e,s} \neq \emptyset$, let $n=|W_{e,s+1}|$.  Assign fresh names, $g_j$, to $(0, 1_1, 1_{n-1}, \ldots , w, 1_{n+1}, \ldots)$ for each $w\in W(n)$, and add them to the domain (note that there will be $p_n^{p_n+1}-1$ such elements).

Let $M_e = \bigcup_s M_{e,s}$. 

\noindent\emph{{End of construction.}}

\medskip

The group $G_e$ is clearly computable.  Moreover, if $e \in \FIN$, $G_e \cong \mathbb{Z} \times W(1) \times \ldots \times W(n)$ where $n = |W_e|$, and is a group of nilpotency class $p_{n}$.

If $e \notin \FIN$,  $G_e \cong \mathbb{Z} \times W(1) \times \ldots \times W(n) \times \ldots$, which is residually nilpotent, but  not nilpotent.
\end{proof}

\begin{theorem}
Detecting a solvable group is $\Sigma^0_2$-complete in the class of computable groups.
\end{theorem}

\begin{proof}

The proof is essentially identical to the nilpotence proof except that rather than using the finite groups $W(n)$ in the construction, we use the uniformly computable \emph{free solvable} groups $H_n = F_2/F_2^{(n)}$, where $F_2$ is the free group on two generators, and $F_2^{(n)}$ is the $n$th group in its derived series.  These groups are infinite, so the construction in this case requires routine dovetailing, and we spare the reader the details.

\end{proof}

\section{Orderability}\label{ord}

We now turn our attention to orderability properties, as this was a main motivation at the beginning of this project.

We say a group $G$ is \textit{partially left-ordered} by relation $\preceq$ if the ordering relation is invariant under the left action of the group on itself.  Formally, if for all $a, g$, and $h$ in $G$ $$g\preceq h \to ag \preceq ah.$$

The group is simply \emph{left-ordered} when $\preceq$ is a total order on $G$.  Similarly, the group is \emph{partially bi-ordered} when $\preceq$ is invariant under multiplication from both the left and right, and \emph{bi-ordered} when the ordering is total.  

Every group has a trivial partial order (equality), and any group $G$ with a non-torsion element $a$ admits a non-trivial partial order which includes the chain  $$\cdots \preceq a^{-2} \preceq a^{-1} \preceq 1_G \preceq a \preceq a^2 \preceq \cdots.$$  Note that torsion elements cannot be ordered relative to the group identity.  For instance if $a \neq 1_G$ is an element of order 3, then from $1_G \preceq a$, it follows that $1_G\preceq a  \preceq a^2 \preceq 1_G$, and we deduce that $a=1$.

Every order, partial or total, is equivalently described by its \textit{upper cone}, the set of elements greater than the identity under the ordering, since $x\preceq y $ if and only if $1_G \preceq x^{-1}y$.  A subset $P$ of a group $G$ is the upper cone of a left-partial order if it is a subsemigroup of $G$ and for all non-identity $g\in G$, we have $g\in P \to g^{-1}\not\in P$.  Moreover, $P$ is the upper cone of a partial bi-order if it is a \textit{normal} subsemigroup, i.e., for all $g\in G$, $gPg^{-1}\subseteq P$.  Finally, $P$ is the upper cone of a total order if for each $g\in G$, either $g$ or its inverse is in $P$.

For more on the theory of ordered and orderable algebraic structures, see \cite{Fuchs}, \cite{KM1996}, and \cite{MR77}.

\medskip

\subsection{Partial orderability and torsion}\label{po_sec}

Every group is trivially partially orderable with the upper cone containing only the identity element.  The existence of a \emph{non-trivial} partial ordering is equivalent to the existence of a non-torsion element in the group.  

\begin{cor} Determining whether a group admits a non-trivial partial (left- or bi-) ordering of its elements is $\Sigma^0_2$-complete in the class of recursively presented groups.

\end{cor}
\begin{proof}
In Corollary \ref{cor_rpm}, it was noted that being a torsion group is $\Pi^0_2$-complete in the class of r.p.~groups, and this corollary follows immediately.
\end{proof}

\begin{cor}

The index set of groups admitting a non-trivial partial  (left- or bi-) ordering is $\Sigma^0_2$-complete in the class of computable groups.

\end{cor}

\begin{proof}
This observation follows immediately from Theorem \ref{thm_tor_cg}.
\end{proof}

\subsection{Bi-orderability}\label{bo_sec}

We consider now the question of total orderability.  There is a condition attributable independently to Ohnishi, Lo\'{s}, and Fuchs \cite{Ohnishi1952, Los1954, Fuchs1958} which provides necessary and sufficient \emph{first order} conditions for orderability.

\begin{theorem}[Ohnish (1952), \L o\'{s} (1954), Fuchs (1958)]

Group $G$ is left-orderable if and only if for every finite sequence  $(g_1,\ldots, g_n)$ of non-identity elements of $G$, there is a sequence $(\epsilon_1, \ldots, \epsilon_n)\in \{\pm 1\}^n$ so that the subsemigroup of $G$ generated by $\{g_1^{\epsilon_1}, \ldots, g_{n}^{\epsilon_{n}}\}$ does not contain the identity $1_G$ of $G$.

\end{theorem}

If we write $\vec{g}$ and $\vec{\epsilon}$ for these sequences, and write $\sgr(\vec{g}^{~\vec{\epsilon}})$ for the subsemigroup generated by $\{g_1^{\epsilon_1}, \ldots, g_{n}^{~\epsilon_{n}}\}$, it is easier to see the complexity of this condition in terms of the arithmetical hierarchy:  

\begin{equation}\label{OLF}\forall \vec{g} \in (G-{\{1_G\}})^{<\omega} ~ \exists \vec{\epsilon}\in \{\pm 1\}^{|\vec{g}|} ~ 1_G\not\in \sgr(\vec{g}^{~\vec{\epsilon}}).
\end{equation}

\noindent Since the existential quantifier is bounded, it may be disregarded. Moreover, if equality is computable in the group, the subsemigroup $\sgr(\vec{g}^{\vec{\epsilon}})$ can be algorithmically enumerated and computably checked against the identity, so the matrix of the formula is equivalent to a $\Pi^0_1$ formula.  When the word problem in group $G$ is computable, the scope of the first universal quantifier, $G-\{1_G\}$, is a computable set.  Hence, the index set of groups that are left-orderable is indeed $\Pi^0_1$ for computable groups.

Exactly the same formula applies to describe the class of bi-orderable groups with only one modification:  Rather than requiring that the subsemigroup $\sgr(\vec{g}^{~\vec{\epsilon}})$ not include the identity, it is required that the \emph{normal} subsemigroup not include the identity.  We will write $S(\vec{g}^{~\vec{\epsilon}})$ for the normal subsemigroup generated by $\vec{g}^{~\vec{\epsilon}}$.

It is natural to ask if it is ``easier'' to determine whether a group admits a bi-ordering if we know already that it admits a left-ordering.  We show that in both the class of computable groups and the class of r.p.~groups, the answer is that it is \textit{not} easier.

\begin{theorem}
The set of indices of groups that admit a bi-ordering is $\Pi^0_1$-complete in the class of computable left-orderable groups.  
\end{theorem}

\begin{proof}  

The detection problem is in $\Pi^0_1$ by virtue of formula (\ref{OLF}) above.

For $\Pi^0_1$-hardness, we construct for each $e\in \nats$, a group $G_e$ which is bi-orderable if and only if $W_e$ is empty, and left-orderable in any case.  Recalling that the index set of the empty set is $\Pi^0_1$-complete, we can accomplish this as follows:  Begin enumerating reduced words on two generators, $a$ and $b$, and their inverses.  At stage $s$, the approximation to the universe of $G_e$ should include all words on $a,b,a^{-1},b^{-1}$ of length less than or equal to $s$, and the approximation to the multiplication table should include all the reduced free products of these having length less than or equal to $s$.  

The first time an element enters $W_e$ at stage $t$, declare the product $a^tb^{-t} = ()$.  In subsequent stages, continue to enumerate the diagram of an isomorphic copy of $\langle a,b~|~a^t=b^t\rangle$.  As a one-relator group, it has decidable word problem, and the rest of the diagram can be effectively determined \cite{Magnus1932}.

If no element ever enters $W_e$, we'll have built a copy of the free group on two generators, which is bi-orderable (though certainly not trivially so\footnote{One can describe a  bi-ordering via a Magnus expansion.  To bi-order $F_2=\langle a,b ~|~\rangle$, consider the ring of formal power series in non-commuting variables $X_a$ and $X_b$, $\ints[X_a,X_b]$.  The map induced by \begin{equation*}\begin{array}{l}

a \to 1+X_a \\
a^{-1} \to 1 - X_a + X_a^2 -X_a^3 +\cdots  \\

b\to 1+ X_b \\

b^{-1} \to 1 - X_b + X_b^2 -X_b^3 +\cdots 

\end{array}
\end{equation*} is an injective homomophism from $F_2$ to $\ints[X_a,X_b]$, the image of which is the multiplicative subgroup generated by $F'_2=\{1+ p(X_a, X_b) ~|~$each term in $p(X_a, X_b)$ has degree at least 1.$\}$.  One can order $\ints[X_a,X_b]$ by writing each power series in a standard form: write the terms in increasing degree, and within each degree, order the terms lexicographically according to subscripts.  To compare two series, compare the coefficients of the first term on which they differ.  The ordering inherited by the subgroup $F'_2$ pulls back to an ordering on $F_2$ via the isomorphism described above.}). If $W_e\neq \emptyset$, $G_e$ is a group with presentation $\langle a,b~|~a^n=b^n\rangle$, which, though torsion-free and left-orderable, is not bi-orderable for bi-orderable groups must have unique roots.  

\end{proof}

\begin{theorem}

Identification of bi-orderability is $\Pi^0_2$-complete in the class of recursively presented left-orderable groups.

\end{theorem}

\begin{proof}

The formula in the previous proof does not suffice to put the problem in $\Pi^0_1$ as the scope of the opening quantifier, $G-\{1_G\}$, is not a computable set (it is co-c.e.).  Suppose $G$ is recursively presented as $\bra x_0, x_1, \ldots ~|~ R_0, R_1, \ldots\ket$.    As before, we write $S(\vec{g}^{~\vec{\epsilon}})$ for the normal subsemigroup generated by $\{g_1^{\epsilon_1}, \ldots, g_{n}^{~\epsilon_{n}}\}$.  This is also a recursively enumerable set, so let $S(\vec{g}^{~\vec{\epsilon}})_s$ be its $s$th finite approximation.

In prenex normal form, with quantifiers over computable sets only, and computable matrix, we give a $\Pi^0_2$ characterization of bi-orderability of recursively presented groups with the formula $$\forall \vec{g}\in F_G^{<\omega} ~ \exists \vec{\epsilon}\in \{\pm 1\}^{|\vec{g}|} ~ \forall t\in \nats ~ \exists s\in \nats ~(\vec{g}\cap 1_{G,s} =\emptyset \to 1_{G,t}\cap S(\vec{g}^{\vec{\epsilon}})_t = \emptyset).$$

To establish completeness, we describe an effective reduction procedure that yields for any $e\in \nats$, a recursive presentation $P_e$ of a group that is bi-orderable if and only if $e\in \INF$.

\medskip

\noindent\emph{Construction.}

\noindent \textit{Stage 0.}  Let $P_0 = \bra x_0, y_0,  x_1, y_1,  \ldots ~|~ x_0^2y_0^{-2} \ket$. 

\smallskip

\noindent \textit{Stage $s+1$.}  If $W_{e,s+1} - W_{e,s} \neq \emptyset$, then let $n=|W_{e,s+1}|$ and add $x_n$, $y_n$, and $x_{n+1}^2y_{n+1}^{-2}$ to the set of relators of $G_s$ to obtain $G_{s+1}$. 

If $W_{e,s+1} - W_{e,s} = \emptyset$, do nothing, and proceed to the next stage.

\noindent \emph{End of construction.} 

\medskip

If $e\not\in \INF$, the resulting group has presentation $$\bra x_0, y_0,  x_1, y_1,  \ldots ~|~ x_0^2y_0^{-2}, \ldots x_n^2y_n^{-2}, x_0, y_0, \ldots x_{n-1}, y_{n-1} \ket,$$ where $n = |W_e|$.  The group itself is the free product of a free group on infinitely many generators ($\bra x_{n+1}, y_{n+1}, \ldots ~|~\ket$) with the group with presentation $\bra x_n, y_n ~|~ x_n^2y_n^{-2}\ket$, which does not have unique roots.  The group is, however, left-orderable.   If $e\in \INF$, the resulting group is trivial, so bi-orderable.  The theorem follows.

%The same ``trick'' with the extra generator ($t$) as in the proof of Theorem \ref{nontorsion} may be used here to obtain a recursive (rather than r.e.) presentation.

%To see hardness, we can make essentially the same reduction as in the previous proof:  At stage $s=0$, the presentation is just that of $F_2 = \langle a,b~|~\rangle$.  If at some stage, $W_{e,s}$ is non-empty, we simply add the relator $a^2=b^2$.  

\end{proof}

\end{document}